\documentclass[12pt]{amsart}
\usepackage[latin1]{inputenc}
\usepackage{indentfirst}
\usepackage{amsfonts}
\usepackage{graphicx}
\usepackage{amsmath}
\usepackage{amssymb}
\usepackage{latexsym}
\usepackage{amscd}
\usepackage[all]{xy}
\usepackage[normalem]{ulem}
\usepackage{lineno}

\usepackage{hyperref}
\usepackage{xcolor}
\usepackage[linesnumbered,lined,ruled,commentsnumbered]{algorithm2e}

\newcommand{\F}{\mathbb {F}}

\newtheorem{theorem}{Theorem}[section]
\newtheorem{definition}[theorem]{Definition}
\newtheorem{lemma}[theorem]{Lemma}
\newtheorem{corollary}[theorem]{Corollary}
\newtheorem{proposition}[theorem]{Proposition}
\newtheorem{remark}[theorem]{Remark}

\def \ord {\operatorname{ord}}
\def \Ord {\operatorname{Ord}}

\begin{document}

\title[Effective existence of primitive $k$-normal elements]{Effective existence of primitive $k$-normal elements via affine spaces}

\author{Josimar J. R. Aguirre, Sarah F. M. Mazzini and Victor G.L. Neumann}

\maketitle

\vspace{1ex}
\small{Instituto de Matem\'{a}tica e Estat\'{i}stica, Universidade Federal de Uberl\^{a}ndia, 
	Av. J. N. \'{A}vila 2121, 38.408-902 Uberl\^{a}ndia -MG, Brazil }

\vspace{8ex}
\noindent
\textbf{Keywords:} Finite fields, $k$-normal elements, primitive elements, affine spaces, character sums, polynomials free of binomials\\
\noindent
\textbf{MSC:} 12E20, 11T30, 11T24.

\begin{abstract}
Let $\mathbb{F}_{q^n}$ be the finite field with $q^n$ elements. An element $\alpha \in \mathbb{F}_{q^n}$ is $k$-normal over $\mathbb{F}_q$ if $\alpha$ and its conjugates span an $\mathbb{F}_q$-subspace of dimension $n-k$. It is known that, for $q$ larger than a constant $c(n)$, a primitive $k$-normal element exists if and only if there exists a $k$-normal element free of binomials, but $c(n)$ is not explicit. We count the $k$-normal elements in affine subspaces of the form $\beta+\ker L_h$ and show that, when $h$ is an irreducible factor of the $\mathbb{F}_q$-order of $\beta$, all but at most one of them are $k$-normal.  We also show that, when $\beta$ is free of binomials, these affine subspaces have degree $n$, so that character sum estimates over affine spaces apply to them. Combining these facts with a sieve, we make the equivalence above effective, with explicit lower bounds on $q$. As an application, we compute these bounds for $5\le n\le 8$ and give a complete characterization of the existence of primitive $k$-normal elements in $\mathbb{F}_{q^6}$ over $\mathbb{F}_q$.
\end{abstract}


\section{Introduction}

Let $\mathbb{F}_{q^n}$ be a finite field with $q^n$ elements, where $q$ 
is a prime power and $n$ is a positive integer.
The field $\mathbb{F}_{q^n}$ has interesting structures related to the two basic field operations.
Regarding the multiplicative structure, $\mathbb{F}_{q^n}^*$ is cyclic and any generator $\alpha \in \mathbb{F}_{q^n}^*$ is a \textit{primitive element}.
On the other hand, in the additive structure, we can consider $\mathbb{F}_{q^n}$ as an $\mathbb{F}_q$-vector space. 
If for an element $\alpha \in \mathbb{F}_{q^n}$ the set $B_{\alpha} = \{ \alpha^{q^i} \mid 0 \leq i \leq n-1 \}$ forms an $\mathbb{F}_q$-basis of $\mathbb{F}_{q^n}$, then $B_\alpha$ is called a normal basis and $\alpha$ is called a \textit{normal element} over $\mathbb{F}_q$.

Both structures are relevant in cryptography and coding theory. Primitive elements
appear in discrete logarithm based cryptography, starting with the Diffie--Hellman key
exchange~\cite{diffie}, and their minimal polynomials are the primitive polynomials, which give
rise to the linear recurring sequences of maximal period \cite[Chapter~8]{Nied}. In a normal
basis, raising to the power $q$ is a cyclic shift of coordinates, and normal bases of low
multiplication complexity, such as optimal normal bases~\cite{GL92,MOVW}, lead to efficient
implementations of finite field arithmetic (see~\cite{gao} for an overview of normal elements
and their applications).

Elements that are simultaneously primitive and normal relate the multiplicative structure
of $\mathbb{F}_{q^n}$ to its structure as a vector space over $\mathbb{F}_q$. The
\textit{Primitive Normal Basis Theorem} states that every extension $\mathbb{F}_{q^n}$ of
$\mathbb{F}_q$ has a normal basis consisting of primitive elements. After partial results by
Carlitz~\cite{carlitz} and Davenport~\cite{davenport1}, it was proved in full generality by
Lenstra and Schoof~\cite{lenstra}, using a combination of additive and multiplicative
character sums, sieving results, and a computer search. Several variations of
this theorem have since been studied (see, for instance, \cite{KR19} and the references
therein).

In 2013, Huczynska et al. \cite{knormal} introduced the concept of $k$-normal elements as an extension of the usual definition of normal elements.
\begin{definition}
Let $\alpha \in \mathbb{F}_{q^n}$. We say that $\alpha$ is a $k$-normal element of $\mathbb{F}_{q^n}$ over $\mathbb{F}_q$
if the set $B_{\alpha} = \{ \alpha^{q^i} \mid 0 \leq i \leq n-1 \}$
generates a vector subspace of $\mathbb{F}_{q^n}$ of dimension $n-k$ over $\mathbb{F}_q$.
\end{definition}
From the above definition, normal elements are exactly the $0$-normal elements, so the case
$k=0$ is the Primitive Normal Basis Theorem~\cite{lenstra}. As observed
in~\cite{knormal}, $k$-normal elements arise implicitly in the construction of the quasi-normal
bases of~\cite{Negre}, which allow efficient multiplication in $\mathbb{F}_{q^n}$. The existence of primitive
$k$-normal elements was settled for $k=1$ in~\cite{lucas1} and for $k=2$ in~\cite{AN}.
General results were obtained
in~\cite{lucas} with effective bounds for $k < n/2$ (see also~\cite{rani} and the
references therein).

In~\cite{lucas0}, Reis obtained bounds for character sums over affine subspaces of
$\mathbb{F}_{q^n}$ and showed that, for $q$ larger than a constant $c(n)$, a primitive
$k$-normal element exists if and only if there exists a $k$-normal element free of binomials
\cite[Theorem~4.8]{lucas0}. In his proof, the character sum bound is applied to the space
$\ker L_g$, where $g$ is the $\mathbb{F}_q$-order of such an element, and $c(n)$ is not
explicit only because of the asymptotic estimates used for $W(q^n-1)$ and $\varphi(q^n-1)$.
Brochero Mart\'{\i}nez, Reis and Ribas~\cite{BRR25} then studied the polynomials that are free
of binomials. They reduced the existence of a $k$-normal element free of binomials to a
condition on the orders and degrees of the divisors of $x^n-1$ \cite[Theorem~1.4]{BRR25}, and
described the possible degrees explicitly when every prime factor of $n$ divides $p(q-1)$,
where $p$ is the characteristic of $\mathbb{F}_q$ \cite[Theorem~1.5]{BRR25}.

In this paper we make the equivalence of \cite[Theorem~4.8]{lucas0} effective. 
Our approach
relies on two properties of the affine spaces $\beta+\ker L_h$, where $h$ is an irreducible factor of
the $\mathbb{F}_q$-order of $\beta$. First, we count the $k$-normal elements in such spaces
(Theorem~\ref{principal}) and show that all but at most one of them are $k$-normal
(Corollary~\ref{cor:k-normal}). Second, we show that, when $\beta$ is free of binomials, these
spaces have degree $n$ in the sense of~\cite{lucas0} (Lemma~\ref{lem:degree}), so that the
character sum bound of~\cite{lucas0} applies to them. Combining these facts with a sieve, we
show that $\beta+\ker L_h$ contains a primitive $k$-normal element for every $q$ above an
explicit bound (Theorem~\ref{teo-principal} and Proposition~\ref{crivo}). Compared
with~\cite{lucas0}, which works with the whole space $\ker L_g$, the spaces $\beta+\ker L_h$
contain at most one element that is not $k$-normal, and they reduce the final computer search
to $q^{\deg h}$ candidates. When $h=x-1$, as in the proof of
Theorem~\ref{principaln6}, the space $\beta+\ker L_h$ is the translate $\beta+\mathbb{F}_q$,
and it contains a primitive $k$-normal element as soon as it contains two primitive elements.
Primitive elements in translates $\theta+\mathbb{F}_q$ are a classical subject, going back to
Davenport~\cite{Dav37} and Carlitz~\cite{Car53}.

As applications, we obtain an explicit criterion for the existence of $(n/2)$-normal elements
free of binomials (Lemmas~\ref{propertyA-even} and~\ref{propertyA-odd}), compute explicit
bounds for $5\le n\le 8$ and give a complete characterization of the existence of primitive
$k$-normal elements in $\mathbb{F}_{q^6}$ over $\mathbb{F}_q$ (Theorem~\ref{principaln6}).

Section~\ref{preliminares} collects background material, including the results
of~\cite{lucas0} and~\cite{BRR25} that we use. In Section~\ref{Main Theorems} we count
$k$-normal elements in affine subspaces, and in Section~\ref{cv} we prove the existence
results and the arithmetic characterization for $k=n/2$. In Section~\ref{pc} we
derive explicit bounds for $5 \le n \le 8$ and prove Theorem~\ref{principaln6}. The
algorithms used in the computations
are given in the Appendix.

\section{Preliminaries}\label{preliminares}

Throughout this paper, $q$ denotes a power of a prime $p$, $\mathbb{F}_q$ denotes the finite field with $q$ elements, and $\mathbb{F}_{q^n}$ denotes its extension of degree $n$.

\subsection{Linearized polynomials and the $\mathbb{F}_q$-order}

The notion of $k$-normality depends on the structure of $\mathbb{F}_{q^n}$ as an
$\mathbb{F}_q[x]$-module, given by linearized polynomials.

\begin{definition}
Let $f \in \mathbb{F}_{q^n}[x]$ be a polynomial of the form $f(x)=\sum_{i=0}^r a_ix^i$.
\begin{enumerate}
\item[(a)] The polynomial $L_f(x)=\sum_{i=0}^r a_ix^{q^i}$ is the 
linearized $q$-associate of $f$.
\item[(b)] For $\alpha \in \mathbb{F}_{q^n}$, we set 
$L_f(\alpha)= \sum_{i=0}^r a_i \alpha^{q^i}$.
\end{enumerate}
\end{definition}

The polynomial 
$L_f$ 
can be viewed as a linear transformation $L_f: \mathbb{F}_{q^n} \longrightarrow \mathbb{F}_{q^n}$ over $\F_q$. 

\begin{lemma}[\cite{Nied}, Lemma 3.59]
Let $f,g \in \mathbb{F}_{q}[x]$. The following hold:
\begin{enumerate}
\item[(a)] $L_f(x)+L_g(x)=L_{f+g}(x)$.
\item[(b)] $L_{fg}(x)=L_f \left( L_g(x) \right) = L_g \left( L_f(x) \right)$.
\end{enumerate}
\end{lemma}


\begin{definition}\label{Order}
Let $g \in \F_q[x]$ be a monic polynomial.
We say that an element $\alpha \in \F_{q^n}$ has $\F_q$-order $g$
if $g$ is the lowest degree monic polynomial over $\mathbb{F}_q$ such that $L_g(\alpha)=0$. In this case, we denote it
by $g=\Ord(\alpha)$.
\end{definition}

It is known that the $\F_q$-order of an element $\alpha \in \F_{q^n}$ divides $x^n-1$.
We also have the following equivalences.

\begin{theorem}[\cite{knormal}, Definition 2.2 and Theorem 3.2]\label{equiv-knormal}
Let $\alpha \in \mathbb{F}_{q^n}$ and $p_\alpha(x) =\sum_{i=0}^{n-1} \alpha^{q^i} x^{n-1-i}$.
The following are equivalent:
\begin{enumerate}
\item[(i)] $\alpha$ is $k$-normal over $\F_q$;
\item[(ii)] $\alpha$ has $\F_q$-order of degree $n-k$;
\item[(iii)] $\gcd(x^n-1, p_\alpha)$ has degree $k$.
\end{enumerate}
\end{theorem}

\subsection{Primitive elements in affine spaces}

Let $m$ be a divisor of $q^n-1$. An element $\alpha \in \mathbb{F}_{q^n}^*$ is called
$m$-free if $\alpha$ is not a $d$-th power in $\mathbb{F}_{q^n}^*$ for any divisor
$d > 1$ of $m$. In particular, $\alpha$ is primitive if and only if it is $(q^n-1)$-free.
This notion goes back to Carlitz~\cite{carlitz} and Davenport~\cite{davenport1}, and was
refined by Lenstra and Schoof~\cite{lenstra}.

\begin{proposition}\label{caracteristica-mlivre}
Let $m$ be a divisor of $q^n-1$ and $\theta(m) = \varphi(m)/m$. For every
$\alpha \in \mathbb{F}_{q^n}^*$,
\[
\theta(m) \sum_{d \mid m} \frac{\mu(d)}{\varphi(d)} \sum_{\chi_d} \chi_d(\alpha)
=
\begin{cases}
1 & \textrm{if } \alpha \textrm{ is } m\textrm{-free},\\
0 & \textrm{otherwise},
\end{cases}
\]
where $\chi_d$ runs over the $\varphi(d)$ multiplicative characters of $\mathbb{F}_{q^n}$
of order $d$.
\end{proposition}
\begin{proof}
See \cite[Section~5.2]{knormal} or \cite[Theorem~2.15]{lucas}.
\end{proof}

To estimate the character sums arising from Proposition~\ref{caracteristica-mlivre} over
affine subspaces, we use the following notion from~\cite{lucas0}, which ensures that the
space is not contained in a proper subfield of $\mathbb{F}_{q^n}$.

\begin{definition}[\cite{lucas0}, Definition 3.1]\label{def:degree_n}
An $\mathbb{F}_q$-affine space $\mathcal{A} = u + \mathcal{V} \subseteq \mathbb{F}_{q^n}$,
where $u \in \mathbb{F}_{q^n}$ and $\mathcal{V} \subseteq \mathbb{F}_{q^n}$ is an
$\mathbb{F}_q$-vector space, is said to have degree $n$ if there exist $y \in \mathcal{A}$ and
$z \in \mathcal{V}$ such that $y z^{-1}$ has degree $n$ over $\mathbb{F}_q$. The dimension of
$\mathcal{A}$ is defined by $\dim \mathcal{A} = \dim_{\mathbb{F}_q} \mathcal{V}$.
\end{definition}

The following result from \cite{lucas0} provides an estimate for character sums over affine
spaces of degree $n$.

\begin{theorem}[\cite{lucas0}, Theorem 3.2]\label{theo:character}
Let $\mathcal{A} = u + \mathcal{V} \subseteq \mathbb{F}_{q^n}$ be an $\mathbb{F}_q$-affine
space of degree $n$ and dimension $t \geq 1$, and let $\chi$ be a nontrivial multiplicative
character of $\mathbb{F}_{q^n}$. Then
\[
\left| \sum_{x \in \mathcal{A}} \chi(x) \right| < n \cdot q^{t - 1/2}.
\]
\end{theorem}

Finally, we recall a necessary condition for primitivity, also from~\cite{lucas0}.

\begin{definition}[\cite{lucas0}, Definition 4.7]\label{def:binomials}
An element $\alpha \in \mathbb{F}_{q^n}$ is free of binomials if $f(x) = x^n-1$ is the unique
monic binomial $f \in \mathbb{F}_q[x]$ of degree at most $n$ such that $L_f(\alpha)=0$.
\end{definition}

Note that the condition for an element $\alpha \in \mathbb{F}_{q^n}$ to be free of binomials
is necessary for it to be primitive. Indeed, if this condition does not hold, then there
exists a binomial $g(x)=x^d-\delta \in \mathbb{F}_q[x]$, with $d<n$, such that $L_g(\alpha)=0$,
implying that $\alpha^{(q^d-1)(q-1)}=1$. Therefore, $\alpha$ cannot be a primitive element of
$\mathbb{F}_{q^n}$.

\begin{theorem}[\cite{lucas0}, Theorem 4.8]\label{Theo-Lucas0}
Let $n\geq 2$ be an integer. Then there exists a constant $c(n) > 0$ such
that, for every $q>c(n)$ and every $0 \leq k \leq n -2$, the following are equivalent:
\begin{enumerate}
    \item there exists a $k$-normal element in $\mathbb{F}_{q^n}$ over $\mathbb{F}_q$ that is
    free of binomials;
    \item there exists a $k$-normal element in $\mathbb{F}_{q^n}$ over $\mathbb{F}_q$ that is
    primitive.
\end{enumerate}
\end{theorem}

The arithmetic condition in Theorem~\ref{Theo-Lucas0} was studied in~\cite{BRR25} in terms
of the order of a polynomial. We recall this classical notion.

\begin{definition}[\cite{Nied}, Definition 3.2]\label{def:order}
	Let $f\in\mathbb{F}_q[x]$ be a nonzero polynomial with $f(0)\neq 0$. The least positive integer
	$e$ such that $f$ divides $x^e-1$ is called the \emph{order} of $f$ and is denoted by
	$\ord(f)$.
\end{definition}

The existence of such an integer $e$ is guaranteed by \cite[Lemma~3.1]{Nied}.
Following~\cite{BRR25}, a polynomial $f\in\mathbb{F}_q[x]$ with $f(0)\neq 0$
is \emph{free of binomials} if it divides no binomial $x^d-\delta\in\mathbb{F}_q[x]$ with
$1\le d<\ord(f)$. If $\beta\in\mathbb{F}_{q^n}$ has $\mathbb{F}_q$-order $g$, then
$L_f(\beta)=0$ if and only if $g$ divides $f$. Hence $\beta$ is free of binomials in the sense
of Definition~\ref{def:binomials} if and only if $\ord(g)=n$ and $g$ is free of binomials
(see \cite[Remark~1.3]{BRR25}).

\section{Counting $k$-normal elements in affine subspaces}\label{Main Theorems}

The constant $c(n)$ in Theorem~\ref{Theo-Lucas0} arises from asymptotic arguments and is not
explicitly computed. Motivated by the problem of obtaining an effective version of this
theorem, we construct affine subspaces of $\mathbb{F}_{q^n}$ in which all, or all but one, of
the elements are $k$-normal. This reduces the existence of primitive $k$-normal elements to
the existence of primitive elements in these subspaces, which is treated in
Section~\ref{cv}.

Given $f,h \in \mathbb{F}_q[x]$, we denote by $L_f^{(h)}$ the restriction of the linear
transformation $L_f \colon \mathbb{F}_{q^n} \to \mathbb{F}_{q^n}$ to $\ker L_h$, that is,
$L_f^{(h)} \colon \ker L_h \to \ker L_h$.

Theorem~\ref{principal} gives an explicit formula for the number of $k$-normal elements contained in a particular affine space. To establish it, we first introduce a definition and present several auxiliary results concerning linear transformations arising from linearized polynomials.

\begin{definition}\label{def-W_d}
Let $\beta\in \mathbb{F}_{q^n}$ be an element of $\mathbb{F}_q$-order
$g\in \mathbb{F}_q[x]$,
let $h \in \mathbb{F}_q[x]$ be a squarefree monic divisor of $g$, and
let $r$ be a monic divisor of $h$. Define
\[
W_h(r) = \left\{\beta+u \; \colon \;
u\in\ker L_h
\; \text { and } \;
L_{g/r}(u)=-L_{g/r}(\beta)
\right\}.
\]
\end{definition}
Observe that if $L_{g/r}(\beta) \in \operatorname{im} L_{g/r}^{(h)}$, then there exists 
$u_0 \in \ker L_h$ such that $L_{g/r}(u_0)=-L_{g/r}(\beta)$ and
\[
W_h(r) = \beta + {L_{g/r}^{(h)}}^{-1}(-L_{g/r}(\beta))
=(\beta + u_0) +\ker L_{g/r}^{(h)}.
\]
Therefore, $|W_h(r)|=|\ker L_{g/r}^{(h)}|$.
Otherwise, if $L_{g/r}(\beta) \not\in \operatorname{im} L_{g/r}^{(h)}$, then
$W_h(r) = \emptyset$.

\begin{lemma}\label{W_d}
Under the same hypotheses as in Definition \ref{def-W_d}, if $r_1,\ldots , r_\ell \in \mathbb{F}_q[x]$ are
different monic irreducible divisors of $h$, then
\[
W_{h}(r_1\cdots r_\ell) = \bigcap_{i=1}^\ell W_{h}(r_i).
\]
\end{lemma}
\begin{proof}
Let $\beta + u \in W_{h}(r_1\cdots r_\ell)$. Then $u \in \ker L_h$ and
$L_{g/( r_1\cdots r_\ell)}(u)=-L_{g/( r_1\cdots r_\ell)}(\beta)$. For all
$1\le i \le \ell$, applying 
$L_{ r_1\cdots r_\ell/r_i}$ to both sides, we get 
$L_{g/r_i}(u)= -L_{g/r_i}(\beta)$.

Conversely, suppose that $u \in \ker L_h$ and $L_{g/r_i}(u)= -L_{g/r_i}(\beta)$
for all $1 \le i \le \ell$. Since $r_1,\ldots , r_\ell$ are different monic irreducible polynomials,
there exist $a_1,\ldots , a_\ell \in \mathbb{F}_q[x]$ such that
\[
\sum_{i=1}^\ell a_i \frac{r_1\cdots r_\ell}{r_i} = 1.
\]
Multiplying by $\frac{g}{r_1 \cdots r_\ell}$, we get
\[
\sum_{i=1}^\ell a_i \frac{g}{r_i} = \frac{g}{r_1 \cdots r_\ell}.
\]
Therefore,
\[
L_{g/(r_1 \cdots r_\ell)}(u) = \sum_{i=1}^\ell L_{a_i} \circ L_{g/r_i} (u) 
= - \sum_{i=1}^\ell L_{a_i} \circ L_{g/r_i} (\beta )
=- L_{g/(r_1 \cdots r_\ell)}(\beta). 
\]
\end{proof}

\begin{lemma}\label{W_d2}
Under the same hypotheses as in Definition \ref{def-W_d},  we have
\[
|W_h(r)|=
\begin{cases}
    q^{\deg h- \deg r} & \textrm{if } \gcd(r, \frac{g}{h})=1, \\
    0             & \textrm{otherwise}.
\end{cases}
\]
\end{lemma}
\begin{proof}
Suppose that $r$ is irreducible and assume that 
$r \mid \frac{g}{h}$.
For $u \in \ker L_h$, we have
\[
L_{g/r}(u) 
= L_{g/(rh)} \circ L_{h}(u) 
= 0 \neq - L_{g/r}(\beta).
\]
Thus,
$W_h(r) = \emptyset$.

If $r$ is not irreducible and $\gcd(r, \frac{g}{h})\neq 1$, from Lemma \ref{W_d}, we also get
$W_h(r) = \emptyset$,
since if $r = r_1 \cdots r_{\ell}$, where $r_1, \dots, r_{\ell}$ are irreducible, we have that for some $i \in \{1,\dots,\ell\}$,
$W_h(r_i) = \emptyset.$

Suppose now that $\gcd(r,\frac{g}{h})=1$.
Since $r \mid h$ and $h$ is squarefree,
then $\gcd (r,\frac{h}{r})=1$. Thus,
$\gcd (r, \frac{g}{r} ) = \gcd (r, \frac{g}{h}\cdot  \frac{h}{r} ) =1$.
Therefore, there exist $a_1,a_2 \in \mathbb{F}_q[x]$ such that
\begin{equation}\label{bezout}
a_1  (g/r)+ a_2 r = 1.
\end{equation}
For $u\in \ker L_r^{(h)}$ we have
\[
u = L_{a_1  (g/r) + a_2 r }(u)
= L_{g/r} (L_{a_1}(u)) \in \operatorname{im} L_{ g/r}^{(h)}.
\]
Conversely, for $u \in \ker L_h$ we have 
$L_r \circ L_{ g/r} (u) =  L_{g/h} \circ L_{h}(u)=0$.
It follows that $\ker L_r^{(h)} = \operatorname{im} L_{ g/r}^{(h)}$.
Since $L_{g/r}(\beta) \in \ker L_h$ and
$L_r \circ L_{g/r}(\beta) = L_g (\beta)=0$, then
$L_{g/r}(\beta) \in \ker L_r^{(h)}= \operatorname{im} L_{ g/r}^{(h)}$.
From the observation after Definition \ref{def-W_d},
we have $|W_h(r)|=|\ker L_{g/r}^{(h)}|$.

It remains to count the number of elements of $\ker L_{g/r}^{(h)}$.
From \eqref{bezout}, if $u \in \ker L_{g/r}^{(h)}$ then 
$u = L_r (L_{a_2}(u)) \in \operatorname{im} L_r^{(h)}$. Since
$L_{g/r}^{(h)} \circ L_r^{(h)} = L_g^{(h)} = 0$, it follows that
$\ker L_{g/r}^{(h)} =\operatorname{im} L_r^{(h)}$.
As this result holds for every $g$ that is divisible by $h$, we also get
$\ker L_{h/r}^{(h)} =\operatorname{im} L_r^{(h)}$, which implies
$|\ker L_{g/r}^{(h)} | =q^{\deg h - \deg r}$.
\end{proof}

\begin{theorem}\label{principal}
Let $\beta\in \mathbb{F}_{q^n}$ be an element of $\mathbb{F}_q$-order
$g\in \mathbb{F}_q[x]$ and
let $h \in \mathbb{F}_q[x]$ be a squarefree monic divisor of $g$.
The affine space $\mathcal{A}_h = \beta + \ker L_h$
of dimension $\deg h$ contains 
\[
q^{\deg h} \prod_{i=1}^\ell \left(1 - \frac{1}{q^{\deg r_i}} \right)
\]
$k$-normal elements, where $k = n - \deg g$
and $r_1,r_2,\ldots, r_{\ell}$ are all the monic irreducible factors of $h$
such that $r_i \nmid \frac{g}{h}$ for $1 \le i \le \ell$.
\end{theorem}

\begin{proof}
For any $u\in \ker L_h,$ we have $\operatorname{Ord}(\beta+u)\mid g.$
From Definition \ref{def-W_d} we have that $\operatorname{Ord}(\beta+u) \neq g$
if and only if there exists 
a monic irreducible factor $r\in\mathbb{F}_{q}[x]$
of $h$ such that $\beta +u \in W_h(r)$.
Thus, the set of $k$-normal elements of $\mathcal{A}_h$ is
\[
\beta + \left( \ker L_h \setminus  \bigcup_{i=1}^\ell W_h(r_i) \right),
\]
where $r_1,\ldots,r_\ell$ are as in the statement.
From Lemmas \ref{W_d}, \ref{W_d2} and the inclusion-exclusion principle we get
that the number of $k$-normal elements of $\mathcal{A}_h$ is
\[
q^{\deg h} - \sum_{j=1}^{\ell}\sum_{i_1<\dots<i_j}(-1)^{j-1}\cdot q^{\deg h-\sum_{w=1}^j\deg r_{i_w}}
= q^{\deg h} \prod_{i=1}^\ell \left(1 - \frac{1}{q^{\deg r_i}} \right).
\]
%
\end{proof}

The following corollary is an immediate consequence of the previous theorem.

\begin{corollary}\label{cor:k-normal}
Let $\beta\in \mathbb{F}_{q^n}$ be an element of  $\mathbb{F}_q$-order 
$g \in \mathbb{F}_q[x]$
of degree $n-k$ and
let $h \in \mathbb{F}_q[x]$ be an irreducible monic divisor of $g$.
Then $\mathcal{A}_h = \beta + \ker L_h$ contains exactly $q^{\deg h}$ $k$-normal elements if
$h^2 \mid g$, and exactly $q^{\deg h}-1$ otherwise.
\end{corollary}

The last corollary  establishes that all but at most one element of the affine space $\mathcal{A}_h$ is
$k$-normal. This result enables us to concentrate on the primitivity of such elements. In the next section, we employ the results of Reis \cite{lucas0} to establish conditions under which primitive $k$-normal elements exist within these affine spaces.

\section{Existence of primitive $k$-normal elements in affine spaces}\label{cv}

We now combine Corollary~\ref{cor:k-normal} with the character sum bound of
Theorem~\ref{theo:character}. For this, the affine space $\mathcal{A}_h$ must have degree $n$
in the sense of Definition~\ref{def:degree_n}, which is guaranteed by the next lemma
whenever $\beta$ is free of binomials and $h$ is irreducible.

\begin{lemma}\label{lem:degree}
Let $\beta\in\mathbb{F}_{q^n}$ be free of binomials and let $h\in\mathbb{F}_q[x]$ be a monic
irreducible divisor of $x^n-1$. Then there exists $z\in\ker L_h$ such that $\beta z^{-1}$ has
degree $n$ over $\mathbb{F}_q$. In particular, $\mathcal{A}_h=\beta+\ker L_h$ has degree $n$.
\end{lemma}

\begin{proof}
Let $e=\ord(h)$. Since $h$ divides $x^e-1$, we have $\ker L_h\subseteq\mathbb{F}_{q^e}$.
Moreover, $z^q\in\ker L_h$ for every $z\in\ker L_h$, since $L_h$ has coefficients in
$\mathbb{F}_q$. For each prime $r\mid n$, let $U_r=\ker L_h\cap\beta\,\mathbb{F}_{q^{n/r}}$,
which is an $\mathbb{F}_q$-subspace of $\ker L_h$. If $z\in\ker L_h\setminus\{0\}$ and
$\beta z^{-1}$ does not have degree $n$, then $\beta z^{-1}\in\mathbb{F}_{q^{n/r}}$ for some
prime $r\mid n$, so $z\beta^{-1}\in\mathbb{F}_{q^{n/r}}$, that is, $z\in U_r$. We show that
the subspaces $U_r$, where $r$ runs over the prime divisors of $n$, do not cover $\ker L_h$.

If $e\mid n/r$, then $\ker L_h\subseteq\mathbb{F}_{q^{n/r}}$, and $U_r\neq\{0\}$ would give
$\beta\in\mathbb{F}_{q^{n/r}}$, that is, $L_{x^{n/r}-1}(\beta)=0$, contradicting the
hypothesis. Hence $U_r=\{0\}$ unless $r$ belongs to the set $P$ of primes $r\mid n$ such that
$e\nmid n/r$.

Let $r\in P$ and let $z\in U_r\setminus\{0\}$. If $z^q\in U_r$, then
$z^{q-1}\in\mathbb{F}_{q^{n/r}}$, so $(z^{q^{n/r}-1})^{q-1}=1$ and therefore
$z^{q^{n/r}}=\lambda z$ for some $\lambda\in\mathbb{F}_q^*$. Since
$\beta z^{-1}\in\mathbb{F}_{q^{n/r}}$, we get $\beta^{q^{n/r}}=\lambda\beta$, that is,
$L_{x^{n/r}-\lambda}(\beta)=0$, a contradiction. Thus, $U_r\cap U_r^{q}=\{0\}$, where
$U_r^q=\{z^q\colon z\in U_r\}\subseteq\ker L_h$. Since $\dim U_r^q=\dim U_r$ and
$U_r\oplus U_r^q\subseteq\ker L_h$, which has dimension $\deg h$, we have
$\dim U_r\le(\deg h)/2$.

Finally, suppose that $\ker L_h\subseteq\bigcup_{r\in P}U_r$. Then
$q^{\deg h}\le|P|\,q^{\lfloor(\deg h)/2\rfloor}$, so $q^{\lceil(\deg h)/2\rceil}\le|P|$.
On the other hand, since $h$ is irreducible,
$q^{\deg h}\equiv 1\pmod e$ by \cite[Corollary~3.4]{Nied}. Moreover, $h$ divides $x^n-1$, so
$e\mid n$ by \cite[Lemma~3.6]{Nied}, and therefore every $r\in P$ divides $e$. As the $i$-th
smallest element of $P$ is at least $i+1$, we obtain
\[
q^{\deg h}\le|P|^2\le(|P|+1)!\le\prod_{r\in P}r\le e<q^{\deg h}.
\]
This contradiction completes the proof.
\end{proof}

\begin{theorem}\label{teo-principal}
Let $n \geq 2$ and $0 \leq k \leq n - 2$. Let $\beta \in \mathbb{F}_{q^n}$ be a $k$-normal
element over $\mathbb{F}_q$ that is free of binomials, and let $h$ be a monic irreducible
divisor of its $\mathbb{F}_q$-order $g = \operatorname{Ord}(\beta)$. Then
the number $\mathcal{P}(\mathcal{A}_h)$ of primitive elements in
$\mathcal{A}_h= \beta + \ker L_h$ satisfies
\begin{equation}\label{primitivos}
\mathcal{P}(\mathcal{A}_h) \ge
q^{\deg h} \cdot \frac{\varphi(q^n-1)}{q^n-1}
\left(
1 -\frac{1}{q^{\deg h}} -
\frac{n \cdot (W(q^n-1)-1)}{q^{1/2}}
\right),
\end{equation}
where $W(q^n-1)$ denotes the number of squarefree divisors of $q^n-1$ and $\varphi$ is Euler's
totient function. In particular, if
\[
q^{1/2} \ge n\, W(q^n - 1),
\]
then $\mathcal{A}_h$ contains a primitive $k$-normal element.
\end{theorem}

\begin{proof}
By Lemma~\ref{lem:degree}, $\mathcal{A}_h$ is an affine space of degree $n$, and its
dimension is $\deg h$. Equation~\eqref{primitivos} follows from the proof of
\cite[Proposition~4.1]{lucas0}, using the character sum bound from
Theorem~\ref{theo:character} applied to $\mathcal{A}_h$.

Since $\frac{m}{\varphi(m)} \leq W(m)$ for every positive integer $m$,
if $q^{1/2} \geq n W(q^n-1)$, then
\[
\begin{aligned}
\frac{q^n-1}{\varphi(q^n-1)} \mathcal{P}(\mathcal{A}_h)
& \geq q^{\deg h - 1/2}
  \left[ q^{1/2} -\frac{1}{q^{\deg h - 1/2}} - n W(q^n-1) + n \right] \\
& > q^{1/2}\, (n-1)
  \geq n(n-1) W(q^n-1)
  \geq n(n-1) \frac{q^{n}-1}{\varphi(q^n-1)}.
\end{aligned}
\]
Hence $\mathcal{P}(\mathcal{A}_h) > n(n-1) \geq 2$. Since $h$ is an irreducible monic
divisor of $g$, Corollary~\ref{cor:k-normal} shows that at most one element
of $\mathcal{A}_h$ fails to be $k$-normal. Therefore $\mathcal{A}_h$ contains a primitive
$k$-normal element.
\end{proof}

Theorem~\ref{teo-principal} guarantees the existence of a primitive $k$-normal
element in $\mathcal{A}_h$ whenever there exists a $k$-normal element
$\beta \in \mathbb{F}_{q^n}$ over $\mathbb{F}_q$ that is free of binomials and
$q^{1/2} \geq n W(q^n-1)$.
To improve the bounds arising from this condition, we apply a sieve method to derive sharper lower bounds on $q$. The proof of the next result is similar to that of \cite[Proposition 3.5]{AN}. However, since we require at least two primitive elements in $\mathcal{A}_h$, 
 we include the proof for completeness.

\begin{proposition}\label{crivo}
Let $q$ be a prime power, let $n \ge 2$ be an integer, and let $k$ be an integer with $0 \le k \le n-2$. Let $m$ be a positive divisor of $q^n-1$, and let $p_1,\ldots,p_r$ be prime numbers such that
$\operatorname{rad}(q^n-1) = \operatorname{rad}(m) \cdot p_1 \cdots p_r.$
Suppose that
$\delta = 1 - \sum_{i=1}^r \frac{1}{p_i} > 0$
and set
$\Delta = \frac{r-1}{\delta} + 2.$
If there exists a $k$-normal element
$\beta \in \mathbb{F}_{q^n}$ over $\mathbb{F}_q$ that is free of binomials, and
\[
q^{1/2} \ge n\, W(m)\, \Delta,
\]
then $\mathcal{A}_h=\beta+\ker L_h$ contains a primitive $k$-normal element,
where $h$ is
a monic irreducible divisor of $\operatorname{Ord}(\beta)$.
\end{proposition}
\begin{proof}
By Corollary~\ref{cor:k-normal}, all but at most one element of $\mathcal{A}_h$ is $k$-normal. Thus, it suffices to show that $\mathcal{A}_h$ contains at least two primitive elements.
By Lemma~\ref{lem:degree}, the affine space $\mathcal{A}_h$ has degree $n$ and dimension $\deg h$, and we may
therefore use the character sum bound from Theorem~\ref{theo:character}.

Let $N$ denote the number of primitive elements in $\mathcal{A}_h$. For a divisor $m$ of $q^n-1$, let $N_m$ be the number of elements $\alpha \in \mathcal{A}_h$ that are $m$-free. 

By Proposition~\ref{caracteristica-mlivre}, summing over
$\alpha \in \mathcal{A}_h \setminus \{0\}$ and using that $\chi_d(0) = 0$ for $d > 1$, we have
\[
N_m = \theta(m) \left( |\mathcal{A}_h \setminus \{ 0\} | + \sum_{\substack{d \mid m \\ d > 1}}
\frac{\mu(d)}{\varphi(d)} \sum_{\chi_d} \sum_{\alpha \in \mathcal{A}_h} \chi_d(\alpha) \right).
\]
The character sum bound from Theorem~\ref{theo:character} gives
\[
\left| \sum_{\substack{d \mid m \\ d > 1}} \frac{\mu(d)}{\varphi(d)} \sum_{\chi_d} \sum_{\alpha \in \mathcal{A}_h} \chi_d(\alpha) \right|
\le n \cdot q^{\deg h - 1/2} \sum_{\substack{d \mid m \\ d > 1}} |\mu(d)|
= n \cdot q^{\deg h - 1/2} (W(m) - 1).
\]
Therefore,
\begin{equation}\label{Nm-q}
\left| N_m - \theta(m)\cdot  |\mathcal{A}_h \setminus \{ 0\} | \right| \le  \theta(m)\cdot  n \cdot q^{\deg h - 1/2} (W(m) - 1).    
\end{equation}

From properties
of $m$-free and $mp_i$-free elements, we have
\[
N \ge \sum_{i=1}^r N_{m p_i} - (r-1) N_m =
 \sum_{i=1}^r \left( N_{m p_i} - \theta(p_i) N_m \right) 
 + \delta N_m,
\]
where $\delta = 1 - \sum_{i=1}^r 1/p_i$.
From \eqref{Nm-q} we have
\[
\left| N_{m p_i} - \theta(p_i) N_m \right| \le \theta(m) \theta(p_i) \cdot n \cdot q^{\deg h - 1/2} W(m).
\]
Combining the last two inequalities, we get
\[
N \ge \delta N_m - \theta(m) \cdot n \cdot q^{\deg h - 1/2} W(m) \sum_{i=1}^r \theta(p_i).
\]

Substituting the lower bound for $N_m$ from \eqref{Nm-q}, we get
\begin{align*}
N &\ge \theta(m) \left( \delta \cdot |\mathcal{A}_h \setminus \{ 0\} | + \delta n \cdot q^{\deg h - 1/2} -
n \cdot q^{\deg h - 1/2} W(m) (2 \delta +r -1) \right)\\
&\ge \delta \theta(m) \left( q^{\deg h} -1 - n \cdot q^{\deg h - 1/2} (W(m) \Delta -1) \right).
\end{align*}
Now, suppose that $q^{1/2} \ge n W(m) \Delta$. Then
\begin{align*}
N & \ge  \delta \theta(m) \left( n \cdot q^{\deg h - 1/2} -1\right)
\\
& \ge 
\delta \theta(m)\left(  n^2W(m) \Delta -1 \right) \ge  n^2 -1 >1 ,
\end{align*}
since $\theta(m)W(m) \ge 1$, $\delta \theta(m) < 1$, and
\[
\delta \Delta = r -1 + 2 \left(1 - \sum_{i=1}^r \frac{1}{p_i} \right)
=1 + \sum_{i=1}^r (1 - \frac{2}{p_i}) \ge 1.
\]
Hence $\mathcal{A}_h$ contains at least two primitive elements, which completes the proof.
\end{proof}

The next result shows that in the case \(k=n/2\),
the existence of a $k$-normal element free of binomials is equivalent to a simple polynomial divisibility condition. From this result, we obtain a complete arithmetic characterization of when this condition is satisfied, for even and odd characteristic respectively (see Lemmas~\ref{propertyA-even} and \ref{propertyA-odd}).

\begin{lemma}\label{(A)}
Let $n$ be even and $k = n/2$. The following are equivalent: \\
(a) There exists a $k$-normal element $\beta \in \mathbb{F}_{q^n}$
free of binomials.
\\
(b) There exists a monic polynomial $g \neq x^k -1$ of degree $k$, multiple of $x-1$, that divides $x^n - 1$.
\end{lemma}
\begin{proof}
For (b) $\Rightarrow$ (a),
let $\beta \in \mathbb{F}_{q^n}$ be an element of $\mathbb{F}_q$-order $g$.
Suppose that there exists a binomial $f= x^{k+\ell} - \gamma$, with $0\le  \ell < k$, such that
$L_f(\beta)=0$.
It follows that
$g$ divides $\gcd (f , x^n-1) = \gcd (f, \gamma x^{k-\ell}-1)$. Thus,
$g$ divides $\gamma x^{k-\ell}-1$ implies $\ell =0$ and $g = x^k +1$.
Since $(x-1) \mid g$ then $\mathbb{F}_q$ has characteristic $2$ and $g = x^k -1$, which is a
contradiction. Hence $\beta$ is free of binomials.

For (a) $\Rightarrow$ (b),
let $g\in\mathbb{F}_q[x]$
be the $\mathbb{F}_q$-order of $\beta$.
Then, clearly $g \neq x^k \pm 1$ and also
$\frac{x^n-1}{g} \neq x^k \pm 1$. One of them is a multiple of $x-1$.
Hence, (b) holds.
\end{proof}

\begin{remark}\label{no-existe}
From the proof of (b) $\Rightarrow$ (a),
every $k$-normal element with $k<n/2$ is free of binomials.
Also, note that statements (a) and (b) of the previous lemma do not hold for \( n = 2 \), as the only monic divisors of degree \( 1 \) of the polynomial \( x^2 - 1 \) are \( x - 1 \) and \( x + 1 \). 
Observe also that the nonexistence of primitive $1$-normal elements in quadratic extensions has already been proved in \cite[Theorem 1.3]{lucas1}.
Hence, for the rest of this section, we assume that $n \ge 4$ and $k = n/2$.

\end{remark}


\begin{lemma}\label{propertyA-even}
Let $q=2^m$ and $k=2^s t \ge 2$, with $t$ odd. The polynomial $x^n-1 \in \mathbb{F}_q[x]$ has a
factor $g$ satisfying Lemma~\ref{(A)}(b) if and only if $t>1$ and either
$\ord_t(q)<\varphi(t)$, or $t$ is a power of a prime $r \le 2^s+1$.
\end{lemma}

\begin{proof}
If $t=1$, then $x^n-1=(x+1)^{2^{s+1}}$, whose only monic factor of degree $k=2^s$ is
$(x+1)^{2^s}=x^k-1$. Hence Lemma~\ref{(A)}(b) does not hold. From now on, assume $t>1$.

Suppose first that $\ord_t(q)<\varphi(t)$. By \cite[Theorem 2.47]{Nied},
the cyclotomic polynomial $Q_t(x)$ factors
into $\varphi(t)/\ord_t(q)\ge 2$ distinct monic irreducible polynomials. Let $h_1 \neq h_2$
be two of them. Since $h_2 \mid Q_t$, the polynomial
\[
g = \frac{x^k-1}{h_2}\, h_1
\]
has degree $k$ and divides $(x^k-1)^2 = x^n-1$. Moreover, $x+1$ divides $g$, because
$x+1 = Q_1$ is coprime to $Q_t$, and $g \neq x^k-1$, because $h_1$ has multiplicity
$2^s+1$ in $g$ and $2^s$ in $x^k-1$. Thus $g$ satisfies Lemma~\ref{(A)}(b).

Suppose now that $\ord_t(q)=\varphi(t)$. Then $(\mathbb{Z}/t\mathbb{Z})^*$ is cyclic,
generated by $q$, and since $t>1$ is odd, $t=r^\ell$ for some odd prime $r$ and some
$\ell\ge 1$. Moreover, $\ord_{r^v}(q)=\varphi(r^v)$ for $1\le v\le \ell$, so all the
polynomials $Q_{r^v}(x)$ are irreducible by \cite[Theorem 2.47]{Nied}. Hence every monic
factor of $x^n-1$ of degree $k$ has the form
\[
g_a(x)=(x+1)^{a_0}Q_r(x)^{a_1}\cdots Q_{r^\ell}(x)^{a_\ell},
\qquad 0\le a_i\le 2^{s+1},
\]
with $a_0+\sum_{i=1}^{\ell} a_i\varphi(r^i)=2^s r^\ell$, and $x^k-1$ corresponds to
$a_0=\cdots=a_\ell=2^s$. Setting $b_i=a_i-2^s$, we see that $g_a$ satisfies
Lemma~\ref{(A)}(b) if and only if $(b_0,\dots,b_\ell)$ is a nonzero tuple of integers with
$|b_i|\le 2^s$, $b_0\ne -2^s$ and
\[
b_0+\sum_{i=1}^{\ell} b_i\varphi(r^i)=0.
\]
If $r \le 2^s+1$, the choice $b_0=r-1$, $b_1=-1$ and $b_i=0$ for $i\ge 2$ satisfies these
conditions. If $r>2^s+1$, let $j=\max\{i : b_i\ne 0\}$. Then $j\ge 1$, since $j=0$ would
give $b_0=0$. But
\[
\Bigl|b_0+\sum_{i=1}^{j-1} b_i\varphi(r^i)\Bigr|
\le 2^s\Bigl(1+\sum_{i=1}^{j-1}\varphi(r^i)\Bigr)
=2^s r^{j-1}<(r-1)r^{j-1}=\varphi(r^j)\le |b_j\varphi(r^j)|,
\]
so no such tuple exists. This completes the proof.
\end{proof}

\begin{lemma}\label{propertyA-odd}
Let $q$ be odd and $k=2^s t \ge 2$, with $s \ge 0$ and $t$ odd. The polynomial $x^n-1$ has a
factor $g \in \mathbb{F}_q[x]$ satisfying Lemma~\ref{(A)}(b) if and only if $t>1$, or
$s \ge 2$, or $s=1$ and $q \equiv 1 \pmod 4$.
\end{lemma}
\begin{proof}
If $t>1$, we may choose
$g=(x^{2^s}-1)(x^{2^s (t-1)} - x^{2^s  (t-2)} + \cdots - x^{2^s} + 1)$, which satisfies Lemma \ref{(A)}(b).

Let us now suppose that $t=1$. Consider the factorization of 
$x^{n}-1= x^{2^{s+1}}-1$
into cyclotomic polynomials $$x^{2^{s+1}}-1=\prod_{i=0}^{s+1} Q_{2^i}(x).$$
Note that if $Q_{2^{s+1}}(x)$ is irreducible, then
all cyclotomic polynomials $Q_{2^i}(x)$ for $0\le i \le s$ are also irreducible, and the only monic
factors of $x^{n}-1$ of degree $k$ are
\[
x^k-1 = \prod_{i=0}^{s} Q_{2^i}(x) \quad
\textrm{and}
\quad
x^k + 1 = Q_{2^{s+1}}(x).
\]
Thus, the existence of a polynomial $g$ satisfying Lemma \ref{(A)}(b) depends on the 
factorization of $Q_{2^{s+1}}(x)$.
By  \cite[Theorem 2.47]{Nied} we have that the cyclotomic polynomial $Q_{2^{s+1}}(x)$ is factored
into $\frac{2^s}{\ord_{2^{s+1}}(q)}$ distinct monic irreducible polynomials in $\mathbb{F}_q[x]$ of the same degree 	
$\ord_{2^{s+1}}(q)$. Hence, the polynomial $g$ satisfying Lemma \ref{(A)}(b) exists if and only if $\ord_{2^{s+1}}(q) < 2^s$. 
Since $q$ is odd, $\ord_{2^{s+1}}(q)$ divides the exponent of
$(\mathbb{Z}/2^{s+1}\mathbb{Z})^*$, which is $2^{s-1}$ if $s \ge 2$ and $2$ if $s = 1$.
Hence $\ord_{2^{s+1}}(q) < 2^s$ always holds for $s \ge 2$, while for $s = 1$ it holds if
and only if $\ord_4(q) = 1$, that is, $q \equiv 1 \pmod 4$.
\end{proof}

When $t=1$, that is, when $n=2^{s+1}$, every prime factor of $n$ divides $p(q-1)$, and
Lemmas~\ref{propertyA-even} and~\ref{propertyA-odd} agree with \cite[Theorem~1.5]{BRR25}.
When $t>1$, a prime factor of $t$ need not divide $p(q-1)$. In such cases \cite[Theorem~1.5]{BRR25} does not apply, while
Lemmas~\ref{propertyA-even} and~\ref{propertyA-odd} still do.

In the next section, using Theorem~\ref{teo-principal} and Proposition~\ref{crivo},
we obtain explicit lower bounds on $q$ that guarantee the existence of primitive $k$-normal
elements in $\mathbb{F}_{q^n}$ and, combining them with Lemmas~\ref{(A)}, \ref{propertyA-even}
and~\ref{propertyA-odd}, completely settle the existence problem for $n=6$.

\section{Existence of primitive $k$-normal elements
in $\mathbb{F}_{q^n}$ over $\mathbb{F}_q$ for $n \le 8$}\label{pc}

, the existence of primitive $k$-normal elements for $k<n/2$ was completely
characterized, with effective lower bounds on $q$. For $k \ge n/2$, the existence problem was
characterized in~\cite{lucas0} for sufficiently large $q$, without explicit bounds, and
\cite[Corollary~1.6]{BRR25} made the arithmetic condition explicit when every prime factor of
$n$ divides $p(q-1)$. In this section, we derive effective lower bounds on $q$ for every
$k$-normal element free of binomials, compute them explicitly for $5 \le n \le 8$, and combine
them with the results of Section~\ref{cv}. As a consequence, we
completely settle the existence problem for $n=6$. The cases $n \le 4$ have already
been solved. The following theorem collects the known results.

\begin{theorem}[\cite{lenstra}; \cite{lucas1}, Theorem 1.3; \cite{AN}, Theorem 1.2]\label{thm:known}
Let $n\ge 2$ and $0 \le k \le n-2$. The following hold for
$\mathbb{F}_{q^n}$ over $\mathbb{F}_q$:
\begin{enumerate}
\item If $n=2$, every primitive element is normal.
\item If $n=3$, there exist primitive $k$-normal elements if and only if
$k=0$ or $k=1$.
\item If $n=4$, there exist primitive $k$-normal elements if and only if
$k=0$, $k=1$, or
$k=2 \text{ and } q\equiv1\pmod4$.
\item If $n\ge 5$ and $0\le k\le 2$, there exist primitive $k$-normal
elements if and only if $k=0$, $k=1$, or
$k=2 \text{ and } \gcd(q^3-q,n)\neq 1$.
\end{enumerate}
\end{theorem}

Observe that, by \cite[Lemma~4.6]{lucas0}, if $n \ge 2$ and $k=n-1$, then no primitive $k$-normal element exists, since no $k$-normal element is free of binomials.

We shall employ the sieve method (Proposition \ref{crivo}) to obtain an initial bound on $q$.
Its effectiveness, however, deteriorates as $n$ increases, making it mainly useful for small values of $n$.

\begin{proposition}\label{gen-bound-1}
Let $n \ge 2$, $0 \le k \le n-2$, and let $q$ be a prime power.
Assume that there exists a $k$-normal element
$\beta \in \mathbb{F}_{q^n}$ over $\mathbb{F}_q$ that is free of binomials.
Let $t$ and $u$ be positive real numbers such that $t+u > 2n$ and
$S_{t,u}<1$, where $S_{t,u}$ denotes the sum of the reciprocals of
all prime numbers between $2^t$ and $2^{t+u}$. Let  $r(t,u)$ denote the number
of prime numbers between $2^t$ and $2^{t+u}$, and define
$\Delta_{t,u} = 2 + \frac{r(t,u)-1}{1-S_{t,u}} $ and
\[
A_{t,u} = \prod_{\substack{\ell< 2^t \\ \ell \textrm{ is prime}}} \frac{2}{\sqrt[t+u]{\ell}}.
\]
Then there exists a primitive
$k$-normal element in $\mathbb{F}_{q^n}$ over
$\mathbb{F}_q$ whenever
\[
q \ge \left( n \cdot A_{t,u} \cdot \Delta_{t,u}\right)^{\frac{2(t+u)}{t+u-2n}}.
\]
\end{proposition}
\begin{proof}
Let $p_1,\ldots , p_r$ be all prime factors of $q^n-1$ between $2^t$ and $2^{t+u}$. For the primes $p_1, \ldots , p_r$, let
$\delta$, $m$ and $\Delta$ be the numbers given in Proposition \ref{crivo}. Observe that
$r\le r(t,u)$, $\delta \ge 1 - S_{t,u}$ and $\Delta \le \Delta_{t,u}$. From \cite[Lemma 2.9]{AN},
we have
\[W(m) \le A_{t,u} \cdot m^{1/(t+u)}\le A_{t,u} \cdot q^{n/(t+u)}.
\]
Thus, if $q \ge \left( n\cdot A_{t,u} \cdot \Delta_{t,u}\right)^{\frac{2(t+u)}{t+u-2n}}$, then
\[
q^{1/2} \ge n \cdot A_{t,u}\cdot q^{\frac{n}{t+u}} \cdot \Delta_{t,u} \ge
n W(m) \Delta.
\]
From Proposition \ref{crivo}, this implies the existence of a primitive $k$-normal element in $\mathbb{F}_{q^n}$ over $\mathbb{F}_q$.
\end{proof}

Since a primitive element is free of binomials, Proposition~\ref{gen-bound-1} shows that, for $q$ above this
bound, a primitive $k$-normal element exists if and only if there exists a $k$-normal element
free of binomials.

\begin{table}[ht]
\centering
\caption{Bounds on $q$ from Proposition~\ref{gen-bound-1}.}
\label{tab:t-u}
\begin{tabular}{c c c c}
\hline
$n$ & $t$ & $u$ & Lower bound on $q$ \\
\hline
$5$ & $6.2$ & $10.8$ & $2.046 \cdot 10^{53}$ \\
$6$ & $6.5$ & $11.2$ & $2.004 \cdot 10^{74}$ \\
$7$ & $6.7$ & $11.4$ & $3.564 \cdot 10^{111}$ \\
$8$ & $7.8$ & $13.3$ & $2.833 \cdot 10^{151}$ \\
\hline
\end{tabular}
\end{table}
For each $n\in\{5,6,7,8\}$, we choose suitable values of $t$ and $u$
to obtain an effective lower bound on $q$ from
Proposition~\ref{gen-bound-1}. The corresponding values are listed in
Table~\ref{tab:t-u}.

Applying the initial bounds from Table~\ref{tab:t-u}, the following proposition yields sharper results.

\begin{proposition}\label{sharper}
Let $n \ge 2$, $0 \le k \le n-2$, and let $q$ be a prime power.
Assume that there exists a $k$-normal element
$\beta \in \mathbb{F}_{q^n}$ over $\mathbb{F}_q$ that is free of binomials.
Let $M$ be a positive integer such that $q \le M$. Let $u$ be the largest positive integer such that $\mathcal{P}_u \le M^n-1$, where $\mathcal{P}_u$
is the product of the first $u$ prime numbers. For any non-negative integer $r \le u$, let $S_{u,r}$ be the sum of the reciprocals of the first $r$ prime numbers starting from the
$(u-r+1)$th prime, and 
define $\Delta_{u,r}= \frac{r-1}{1-S_{u,r}}+2$.
Let us define
\[
N = \max_{1 \le u_0 \le u} \{\min_{0 \le r \le u_0} \{(n \cdot 2^{u_0-r} \cdot \Delta_{u_0,r} )^2 \colon S_{u_0,r}<1\}\}.
\]
Then, if $N \le q \le M$ there exists a primitive $k$-normal element in $\mathbb{F}_{q^n}$ over $\mathbb{F}_q$.
\end{proposition}

\begin{proof}
Let $q$ be a prime power with $N \le q \le M$. Let $u_0$ be the number of distinct prime divisors of $q^n - 1$. Since $q^n - 1 \le M^n - 1$, we have $u_0 \le u$. Write the distinct prime divisors of $q^n - 1$ as
$\ell_1 < \ell_2 < \cdots < \ell_{u_0}$.

Choose $r_0 \in \{0,1,\ldots,u_0\}$ with $S_{u_0,r_0} < 1$ such that
\[
(n \cdot 2^{u_0-r_0} \cdot \Delta_{u_0,r_0} )^2 =
\min_{0 \le r \le u_0} \{(n \cdot 2^{u_0-r} \cdot \Delta_{u_0,r} )^2
\colon S_{u_0,r} <1\}.
\]
Define
$m := \ell_1 \ell_2 \cdots \ell_{u_0-r_0}$
and
$p_1 = \ell_{u_0-r_0+1}, \ldots, p_{r_0} = \ell_{u_0}.$
Then
\[
\operatorname{rad}(q^n - 1) = \operatorname{rad}(m) \cdot p_1 \cdots p_{r_0}
\]
and $W(m) \le 2^{u_0-r_0}$.
Since the prime numbers $p_1 = \ell_{u_0-r_0+1}, \ldots, p_{r_0} = \ell_{u_0}$
are greater than or equal to the first $r_0$ prime numbers starting from the
$(u_0-r_0+1)$th prime, then
\[
\delta = 1 - \sum_{i=1}^{r_0} \frac{1}{p_i} \ge  1 - S_{u_0,r_0} > 0,
\]
and
\[
\Delta = 2 + \frac{r_0-1}{\delta} \le 2 + \frac{r_0-1}{1 - S_{u_0,r_0}} = \Delta_{u_0,r_0}.
\]
Therefore,
\[
n W(m) \Delta \le n \cdot 2^{u_0-r_0} \cdot \Delta_{u_0,r_0}.
\]
By the choice of $r_0$ and the definition of $N$, we have
\[
\left( n \cdot 2^{u_0-r_0} \cdot \Delta_{u_0,r_0} \right)^2 \le N \le q.
\]
Taking square roots yields
$q^{1/2} \ge n \cdot 2^{u_0-r_0} \cdot \Delta_{u_0,r_0} \ge n W(m) \Delta$.
Since we have assumed the existence of a $k$-normal element
free of binomials,
Proposition~\ref{crivo} implies that $\mathbb{F}_{q^n}$ contains a primitive $k$-normal element.
\end{proof}
Starting from Table~\ref{tab:t-u} and iterating Proposition~\ref{sharper}, each time taking as $M$ the value of $N$ obtained in the previous step, we obtain the bounds in Tables~\ref{tab:n5},~\ref{tab:n6},~\ref{tab:n7},~\ref{tab:n8}.

\begin{table}[htbp]
\centering
\begin{minipage}{0.48\textwidth}
\centering
\caption{Bounds \(q \ge N\) for \(n=5\)}
\label{tab:n5}
\begin{tabular}{ccc}
\hline
\(q \le M\) & \(u\) & \(q \ge N\) \\
\hline
\(2.046\cdot10^{53}\) & $117$ & \(1.814\cdot10^{10}\) \\
\(1.814\cdot10^{10}\) & $33$ & \(5.325\cdot10^{7}\) \\
\(5.325\cdot10^{7}\) & $27$ & \(1.974\cdot10^{7}\) \\
\(1.974\cdot10^{7}\) & $26$ & \(1.656\cdot10^{7}\) \\
\(1.656\cdot10^{7}\) & $25$ & \(1.387\cdot10^{7}\) \\
\hline
\end{tabular}
\end{minipage}
\hfill
\begin{minipage}{0.48\textwidth}
\centering
\caption{Bounds \(q \ge N\) for \(n=6\)}
\label{tab:n6}
\begin{tabular}{ccc}
\hline
\(q \le M\) & \(u\) & \(q \ge N\) \\
\hline
\(2.004\cdot10^{74}\) & $179$ & \(2.390\cdot10^{11}\) \\
\(2.390\cdot10^{11}\) & $41$ & \(1.916\cdot10^{8}\) \\
\(1.916\cdot10^{8}\) & $32$ & \(6.553\cdot10^{7}\) \\
\(6.553\cdot10^{7}\) & $31$ & \(5.586\cdot10^{7}\) \\
\(5.586\cdot10^{7}\) & $30$ & \(4.756\cdot10^{7}\) \\
\hline
\end{tabular}
\end{minipage}

\end{table}

\begin{table}[htbp]
\centering
\begin{minipage}{0.48\textwidth}
\centering
\caption{Bounds \(q \ge N\) for \(n=7\)}
\label{tab:n7}
\begin{tabular}{ccc}
\hline
\(q \le M\) & \(u\) & \(q \ge N\) \\
\hline
\(3.564\cdot10^{111}\) & $285$ & \(4.188\cdot10^{12}\) \\
\(4.188\cdot10^{12}\) & $49$ & \(5.431\cdot10^{8}\) \\
\(5.431\cdot10^{8}\) & $37$ & \(1.737\cdot10^{8}\) \\
\(1.737\cdot10^{8}\) & $36$ & \(1.561\cdot10^{8}\) \\
\hline
\end{tabular}
\end{minipage}
\hfill
\begin{minipage}{0.48\textwidth}
\centering
\caption{Bounds \(q \ge N\) for \(n=8\)}
\label{tab:n8}
\begin{tabular}{ccc}
\hline
\(q \le M\) & \(u\) & \(q \ge N\) \\
\hline
\(2.833\cdot10^{151}\) & $413$ & \(4.122\cdot10^{13}\) \\
\(4.122\cdot10^{13}\) & $58$ & \(1.482\cdot10^{9}\) \\
\(1.482\cdot10^{9}\) & $43$ & \(4.129\cdot10^{8}\) \\
\(4.129\cdot10^{8}\) & $41$ & \(3.406\cdot10^{8}\) \\
\hline
\end{tabular}
\end{minipage}
\end{table}


Using all the results obtained in this work, we now give a complete characterization
of the existence of primitive $k$-normal elements in $\mathbb{F}_{q^6}$ over $\mathbb{F}_q$
for all $0 \le k \le 4$.

\begin{theorem}\label{principaln6}
Let $q$ be a prime power and $0 \le k \le 4$. There exists a primitive $k$-normal element
in $\mathbb{F}_{q^6}$ over $\mathbb{F}_q$ if and only if one of the following holds:
\begin{enumerate}
    \item[(i)] $0 \le k \le 2$;
    \item[(ii)] $k = 3$ and $q$ is an odd prime power or a power of $2$ with an even exponent;
    \item[(iii)] $k = 4$ and $q \equiv 1 \pmod{6}$.
\end{enumerate}
\end{theorem}

\begin{proof}
The cases $k = 0$, $k = 1$, and $k = 2$ follow directly from Theorem~\ref{thm:known}.
Thus, we assume $k \in \{3,4\}$.
We have the following factorization of $x^6-1$ into irreducibles over $\mathbb{F}_q$:
\[
x^6 - 1 =
\begin{cases}
(x+1)^2(x^2+x+1)^2 & \textrm{if } q=2^{2m+1},\\
(x+1)^2(x+b)^2(x+b+1)^2 & \textrm{if } q=2^{2m},\\
(x-1)^3(x+1)^3 & \textrm{if } 3 \mid q,\\
(x-1)(x+1)(x^2+x+1)(x^2-x+1) & \textrm{if } q \equiv 5 \pmod 6,\\
(x-1)(x+1)(x-b)(x-b-1)(x+b)(x+b+1) & \textrm{if } q \equiv 1 \pmod 6,
\end{cases}
\]
where $b \in \mathbb{F}_q$ satisfies $b^2+b+1=0$.

For $k=4$, the monic divisors of $x^6-1$ of degree $2$ are, in the first four cases,
binomials or divisors of $x^3-1$ or $x^3+1$, so no $4$-normal element is free of binomials.
If $q\equiv 1\pmod 6$, then $g=(x-1)(x+b)$ divides no binomial of degree less than $6$, so
there exist $4$-normal elements free of binomials.

For $k=3$, Lemmas~\ref{propertyA-even} and~\ref{propertyA-odd}, combined with
Lemma~\ref{(A)}, provide a $3$-normal element free of binomials if and only if
$q=2^{2m}$ or $q$ is an odd prime power.

Since the bounds of Propositions~\ref{gen-bound-1} and~\ref{sharper} do not depend on $k$,
Table~\ref{tab:n6} gives a primitive $k$-normal element in $\mathbb{F}_{q^6}$ over
$\mathbb{F}_q$ for every $q \ge 4.756 \cdot 10^{7}$ in these cases.

Using SageMath~\cite{SAGE}, we apply Algorithm~\ref{sieve} (with $n=6$),
which tests the criterion of Proposition~\ref{crivo} taking
$p_1,\ldots , p_r$ to be the $r$ largest prime divisors of $q^6-1$. It fails for exactly $13700$ of the prime powers less than $4.756 \cdot 10^{7}$ that are either odd prime powers
or powers of $2$ with an even exponent.
Among these $13700$ prime powers, $6843$ satisfy $q \equiv 1 \pmod 6$. These are the ones
to be treated in the case $k=4$, while all $13700$ must be treated in the case $k=3$.

\begin{table}[ht]
\centering
\caption{Choice of $f$ in Algorithm~\ref{explicit}.}\label{tab:fpolys}
\begin{tabular}{cll}
\hline
\multicolumn{1}{c}{$k$} & \multicolumn{1}{c}{condition on $q$} & \multicolumn{1}{c}{polynomial $f$} \\
\hline
$3$ & $q=2^{2m}$ or $q \equiv 1 \pmod 6$
    & $(x-b)(x-b-1)(x+b+1)$ \\
$3$ & $3 \mid q$ or $q \equiv 5 \pmod 6$
    & $(x+1)(x^2+x+1)$ \\
$4$ & $q \equiv 1 \pmod 6$
    & $(x+1)(x-b)(x-b-1)(x+b+1)$\\
\hline
\end{tabular}
\end{table}

For these $13700$ prime powers we apply Algorithm~\ref{explicit}, which is based on
Theorem~\ref{teo-principal} with $h = x-1$, and returns a primitive $k$-normal element of the form
$L_f(\alpha)+\delta$, where $\delta \in \mathbb{F}_q$ and $\alpha$ is a normal element
of $\mathbb{F}_{q^6}$ over $\mathbb{F}_q$. Writing $a$ for a primitive element of
$\mathbb{F}_q$ and, whenever $3 \mid (q-1)$, $b = a^{(q-1)/3}$,
the polynomial $f$ is given in Table~\ref{tab:fpolys},
and in each case $g = (x^6-1)/f$ is free of binomials and divisible by $x-1$.
The $6843$ prime powers with $q \equiv 1 \pmod 6$ are treated for both values of $k$,
the remaining ones only for $k=3$. In every case the algorithm returns a primitive
$k$-normal element, which completes the proof.
\end{proof}

\begin{remark}\label{rem:n578}
Let $n\in\{5,7,8\}$ and let $k$ be an integer with $n/2\le k\le n-2$. Whether
$\mathbb{F}_{q^n}$ contains a $k$-normal element over $\mathbb{F}_q$ that is free of binomials
can be decided by an inspection of the factorization of $x^n-1$ over $\mathbb{F}_q$. Since
Theorem~\ref{teo-principal} applies to every such element, the existence problem for
primitive $k$-normal elements can then be settled exactly as in Theorem~\ref{principaln6},
combining the bounds of Tables~\ref{tab:n5}, \ref{tab:n7} and~\ref{tab:n8} with
Algorithm~\ref{sieve} and a suitable adaptation of Algorithm~\ref{explicit}, in which $h$ is
an irreducible factor of least degree of the $\mathbb{F}_q$-order. 
\end{remark}

For $n \ge 9$, Propositions~\ref{gen-bound-1} and~\ref{sharper} give explicit bounds in the
same way, since Theorem~\ref{teo-principal} applies to every $k$-normal element free of
binomials.
For instance, for $n=9$, $k=5$ and $q\equiv -1\pmod 9$, the $\mathbb{F}_q$-order of
such an element is a product of two irreducible quadratic polynomials, and one may take $h$
to be either of them. The computational cost of the remaining cases, however, grows quickly
with $n$.

\bigskip
%
\noindent\textbf{Author contributions}\enspace All authors contributed to the development of
the results, participated in drafting and revising the manuscript, and approved the final
version for submission.

\medskip
\noindent\textbf{Funding}\enspace All three authors were partially supported by FAPEMIG grant
RED-00133-21. The first and third authors were partially supported by CNPq grant
420721/2025-8, and the third author was also partially supported by CNPq grant
307924/2023-8. This study was financed in part by the Coordena\c{c}\~ao de
Aperfei\c{c}oamento de Pessoal de N\'{\i}vel Superior -- Brasil (CAPES) -- Finance Code 001.

\medskip
\noindent\textbf{Data availability}\enspace No datasets were generated or analysed during the
current study.

\section*{Declarations}

\noindent\textbf{Conflict of interest}\enspace The authors declare no conflict of interest.

\medskip
\noindent\textbf{Use of AI tools}\enspace The authors used Claude (Anthropic) to improve the
language and presentation of the manuscript and to review the SageMath code used
in the computations. All suggestions were checked by the authors, and all computations were
carried out and verified by the authors in SageMath. The authors take full responsibility for
the content of the article.

\section*{Appendix: Algorithms} \label{apendice}
The following algorithms, implemented in SageMath~\cite{SAGE}, are used in the proof of Theorem~\ref{principaln6}.
Algorithm~\ref{sieve} checks the sieve criterion of Proposition~\ref{crivo} for a given prime
power $q$ and extension degree $n$. Algorithm~\ref{explicit} constructs a primitive $k$-normal
element of $\mathbb{F}_{q^6}$ over $\mathbb{F}_q$, for $k \in \{3,4\}$, via
Theorem~\ref{teo-principal}. The test $d = k$ relies on Theorem~\ref{equiv-knormal}(iii).

{\scriptsize
\begin{algorithm}
\KwIn{A prime power $q$, a positive integer $n$}
\KwOut{True or False}
LPrime $\gets$ List of the prime divisors of $q^n - 1$ in increasing order \\
$u \gets$ Length(LPrime)\\
Value $\gets$ False\\
$r \gets 0$\\
\While{$r \le u$ and Value $=$ False}
    {
    $S\gets$ List of the last $r$ elements of LPrime\\
    $\delta \gets 1- \displaystyle\sum_{p\in S} \frac{1}{p}$\\
    \If{$\delta >0$}
        {
        $\Delta \gets \frac{r-1}{\delta}+2$\\
        $N \gets (n\cdot 2^{u-r} \cdot \Delta)^2$\\
        \If{$q\ge N$}
            {
            Value $\gets$ True
            }
        }
    $r\gets r+1$
    }
\Return Value
\caption{Sieve($q,n$)}\label{sieve}
\end{algorithm}
}

{\scriptsize
\begin{algorithm}
\KwIn{A prime power $q$ and $k \in \{3,4\}$, with $q \neq 2^{2m+1}$ if $k=3$ and $q \equiv 1 \pmod 6$ if $k=4$}
\KwOut{A primitive $k$-normal element of the extension $\mathbb{F}_{q^6}\mid \mathbb{F}_q$ or False}
$a \gets$ A primitive element of $\mathbb{F}_{q}$\\
$\alpha \gets$ A normal element of $\mathbb{F}_{q^6}$ over $\mathbb{F}_{q}$\\
$\gamma \gets$ False\\
\uIf{$k=4$}
     {
     $b \gets a^{\frac{q-1}{3}}$\\
     $f \gets x^4 - (b-1)x^3 - 2bx^2 - (2b+1)x - b - 1$
     }
\uElseIf{$q=2^{2m}$ or $q \equiv 1 \pmod 6$}
     {
     $b \gets a^{\frac{q-1}{3}}$\\
     $f \gets x^3 - bx^2 - bx - b - 1$
     }
\Else{
     $f \gets x^3 + 2x^2 + 2x + 1$
     }
$\beta \gets L_f(\alpha)$\\
NotEnd $\gets$ True \\
\uIf{$\ord(\beta)=q^6-1$}
    {
     NotEnd $\gets$ False \\
     $\gamma \gets \beta$
    }
\Else{
    $j \gets 0$\\
    \While{$j<q-1$ and NotEnd}
          {\If{$\ord(\beta+a^j)=q^6-1$}
               {
               $p \gets \displaystyle\sum_{i=0}^{5} (\beta+a^j)^{q^i}x^{5-i}$\\
               $d \gets \deg (\gcd(p,x^6-1))$\\
               \If{$d=k$}
                  {NotEnd $\gets$ False \\
                   $\gamma \gets \beta+a^j$}
               }
          $j \gets j+1$
          }
     }
\Return $\gamma$
\caption{Explicit-primitive-$k$-normal($q,k$)}\label{explicit}
\end{algorithm}
}

\end{document}